\documentclass[12pt]{amsart}
\usepackage{amsmath}
\usepackage{latexsym}
\usepackage{amssymb}
\usepackage{amscd}

\usepackage{amsfonts}

\usepackage{mathrsfs}

\usepackage{bbm}
\usepackage{mathabx}
\usepackage{skak}
\usepackage{pigpen}

\newtheorem{theorem}{Theorem}

\newtheorem{definition}{Definition}

\newcounter{tmp}

\def\ZZ{{\mathbb Z}}
\def\PP{{\mathbb P}}

\def\dA{\mathscr A}
\def\dB{\mathscr B}
\def\dC{\mathscr C}
\def\dD{\mathscr D}
\def\dE{\mathscr E}
\def\dF{\mathscr F}

\def\dS{\mathscr S}

\def\mM{\mathsf M}
\def\mN{\mathsf N}
\def\mP{\mathsf P}
\def\mT{\mathsf T}
\def\mS{\mathsf S}
\def\mR{\mathsf R}
\def\mY{\mathsf Y}

\def\ma{\mathsf a}
\def\mb{\mathsf b}

\def\mh{\mathsf h}

\def\mPhi{\mathsf \Phi}

\def\mCone{\mathsf{Cone}}

\def\kk{{\mathbf k}}

\def\dHom{\mathsf{Hom}}
\def\Hom{\operatorname{Hom}}

\newcommand{\SF}{\dS\!\dF\!\operatorname{--}\!}
\newcommand{\SFf}{\dS\!\dF_{fg}\!\operatorname{--}\!}
\newcommand{\prfdg}{\mathscr{P}\!\mathit{erf}\!\operatorname{--}}

\newcommand{\Ac}{\dA\!\mathit{c}\!\operatorname{--}\!}
\newcommand{\prf}{\mathcal{P}\!\mathit{erf}\!\operatorname{--}}
\newcommand{\Ob}{\operatorname{Ob}}
\newcommand{\Ho}{{\H^0}}

\def\Mod{{\mathscr M}\!\mathit{od}\!\operatorname{--}\!}

\def\supp{\operatorname{supp}}
\def\length{\operatorname{length}}

\def\hy{\mbox{-}}

\def\D{{\mathcal D}}

\def\H{{\mathcal H}}

\def\bR{{\mathbf R}}
\def\bL{{\mathbf L}}

\def\wt{\widetilde}

\title[]{Gluing of categories and Krull--Schmidt partners}

\author[]{Dmitri Orlov}

\address{Algebraic Geometry Department, Steklov Mathematical Institute of Russian Academy of Sciences, 8 Gubkin str., Moscow 119991, RUSSIA}

\email{orlov@mi.ras.ru}

\thanks{This work is supported by the Russian Science Foundation under  grant 14-50-00005}

\date{}

\begin{document}

\maketitle
Let $\kk$ be a field and let $\dA$ be a small $\kk$\!--linear differential graded (DG) category. We define a {\it (right) DG $\dA$\!--module} as a DG functor
$\mM: \dA^{op}\to \Mod \kk,$ where $\Mod \kk$ is the DG category of complexes of $\kk$\!--vector spaces.
Let $\Mod \dA$ be the DG
category of right DG $\dA$\!--modules and let $\Ac\dA$ be the full
DG subcategory consisting of all acyclic DG modules.
The
homotopy category $\Ho(\Mod\dA)$ is triangulated
and $\Ho (\Ac\dA)$ forms a full triangulated subcategory.
The {\it derived
category} $\D(\dA)$ is defined as the Verdier quotient
$
\D(\dA):=\Ho(\Mod\dA)/\Ho (\Ac\dA).
$

Each object $\mY\in\dA$ defines the representable DG module
$
\mh^\mY_{\dA}(-):=\dHom_{\dA}(-, \mY).
$
The DG module is called {\it free} if it is isomorphic to a direct sum of  DG modules of the form
$\mh^\mY[m].$
A DG module
$\mP$ is called {\it semi-free} if it has a filtration
$0=\mPhi_0\subset \mPhi_1\subset ...=\mP$
with free quotient  $\mPhi_{i+1}/\mPhi_i.$
The full
DG subcategory of semi-free DG modules is denoted by $\SF\dA.$
The canonical DG functor $\SF\dA\hookrightarrow\Mod\dA$ induces an equivalence
$\Ho(\SF\dA)\stackrel{\sim}{\to} \D(\dA)$ of the triangulated categories
\cite{Ke, LO}.
We denote by $\SFf\dA\subset \SF\dA$ the full DG subcategory of finitely generated semi-free
DG modules, i.e. $\mPhi_n=\mP$ for some $n$ and $\mPhi_{i+1}/\mPhi_i$ is a finite direct sum of
$\mh^Y[m].$ A {\em DG category of perfect modules} $\prfdg\dA$
is the full DG subcategory of $\SF\dA$ consisting of all DG modules that are homotopy
equivalent to  direct summands of object of $\SFf\dA.$
Denote by $\prf\dA$ the homotopy category $\Ho(\prfdg\dA).$ It is triangulated
and it is equivalent to the subcategory of compact objects $\D(\dA)^c\subset \D(\dA).$
If $\dA$ is pretriangulated and $\Ho(\dA)$ is idempotent complete, then
the Yoneda functor $\mh: \dA\to\prfdg\dA$ is a quasi-equivalence.

\begin{definition}\label{upper_tr}
Let $\dA, \dB$ be small DG categories and let $\mS$ be a $\dB\hy\dA$\!--bimodule.
Define an {\em upper triangular} DG category $\dC=\dA\underset{\mS}{\with}\dB$ corresponding to  the data
$(\dA, \dB, \mS)$ as follows:
\begin{enumerate}
\item[1)] $\Ob(\dC)=\Ob(\dA)\bigsqcup\Ob(\dB),$

\item[2)]
$
\dHom_{\dC}(X, Y)=
\begin{cases}
 \dHom_{\dA}(X, Y)\; \text{or}\quad \dHom_{\dB}(X, Y),& \text{ when} \; X,Y\in\dA\; \text{or}\quad X,Y\in\dB\\
      \mS(Y, X), & \text{ when $X\in\dA, Y\in\dB$}\\
      0, & \text{ when $X\in\dB, Y\in\dA$}
\end{cases}
$
\end{enumerate}
with the composition law coming from $\dA,\dB$ and the bimodule structure on $\mS$ (see \cite{Ta, O_glue}).
\end{definition}

\begin{definition}\label{gluing_cat}
Let $\dA$ and $\dB$ be  small DG categories and let $\mS$ be a $\dB\hy\dA$\!--bimodule.
 A {\em gluing} $\prfdg\dA\underset{\mS}{\oright}\prfdg\dB$ (resp. $\D(\dA)\underset{\mS}{\oright}\D(\dB)$) via $\mS$ is defined as $\prfdg\dC$
 (resp. $\D(\dC)$),
 where $\dC=\dA\underset{\mS}{\with}\dB.$
\end{definition}

The natural inclusions $\ma: \dA\hookrightarrow \dA\underset{\mS}{\with}\dB$ and
$\mb: \dB\hookrightarrow \dA\underset{\mS}{\with}\dB$ define the fully faithful DG
functors $\ma^*$ and $\mb^*$ from $\prfdg\dA$ and $\prfdg\dB$  to $\prfdg\dA\underset{\mS}{\oright}\prfdg\dB,$
which induce fully faithful functors $a^*, b^*$ from $\D(\dA), \D(\dB)$ (resp. $\prf\dA, \prf\dB$) to
$\D(\dC)$ (resp. $\prf\dC$) and give semi-orthogonal decompositions of triangulated categories
$
\D(\dC)=\langle\D(\dA),\; \D(\dB)\rangle$ and
$
\prf\dC=\langle\prf\dA,\; \prf\dB\rangle.
$
The natural restriction functors $\ma_*, \mb_*$ from $\SF(\dA\underset{\mS}{\with}\dB)$ to $\SF\dA$ and $\SF\dB$
induce derived functors $a_*, b_*$ from $\D(\dA\underset{\mS}{\with}\dB)$ to $\D(\dA)$ and $\D(\dB)$
that send perfect modules to perfect modules.

Let $\mT$ be a
$\dB\hy\dA$\!--bimodule.
For each DG $\dB$\!--module $\mN$ we can construct a DG $\dA$\!--module
$\mN\otimes_{\dB} \mT.$
The DG functor $(-)\otimes_{\dB} \mT: \Mod\dB \to \Mod\dA$ does not respect
quasi-isomorphisms, in general. However, if $\mT$ is semi-free $\dB\hy\dA$\!--bimodule
we obtain a DG functor $(-)\otimes_{\dB} \mT: \SF\dB \to \SF\dA$ that on homotopy level
induces the derived functor
$(-)\stackrel{\bL}{\otimes}_{\dB}\mT: \D(\dB)\to\D(\dA).$
In the case when the functor $(-)\stackrel{\bL}{\otimes}_{\dB}\mT$ sends perfect modules to perfect modules,
we also obtain a DG functor $(-)\otimes_{\dB} \mT: \prfdg\dB \to \prfdg\dA.$

Any $\dB\hy\dA$\!--bimodule $\mP$ gives a $\dB\hy\dC$\!--bimodule
$\underline{\mP}$ such that $\underline{\mP}(B, C)=\mP(B, C),$ when $C\in \dA,$ and is $0$ otherwise.
Any morphism $\phi: \mS\to \mT$ of $\dB\hy\dA$\!--bimodules induces a $\dB\hy\dC$\!--bimodule
$\wt\mT$ by the rule
$\wt{\mT}(B, C)=
 \mT(B, C),$ when $C\in\dA,$ and
$\wt{\mT}(B, C)=\dHom_{\dB}(C, B),$ when $C\in\dB,$ with a natural bimodule structure
taking in account the composition $\dHom_{\dB}(C, B)\otimes \mS(C, A)\to \mS(B, A)\stackrel{\phi}{\to}\mT(B, A)$ with $C\in\dB.$

There is an isomorphism $\wt{\mS}(B, -)\cong\dHom_{\dC}(-,B)=\mh^{B}_{\dC}$ and the functor
$(-)\stackrel{\bL}{\otimes}_{\dB}\wt{\mS}: \D(\dB)\to\D(\dC)$ is isomorphic to  the fully faithful functor $b^*: \D(\dB)\to\D(\dC).$

\begin{theorem}\label{main1} Let DG categories $\dA, \dB,$ DG $\dB\hy\dA$\!--bimodules $\mS, \mT,$ and a morphism $\phi:\mS\to\mT$ be as above.
Let $\mR=\mCone(\phi)$ be the cone of $\phi.$
Suppose that for any  DG  $\dB$\!--modules $\mM$ and $\mN$ the following condition holds
\begin{equation}\label{cond}
\Hom_{\D(\dA)}(\mM\stackrel{\bL}{\otimes}_{\dB}\mR,\; \mN\stackrel{\bL}{\otimes}_{\dB}\mT)=0.
\end{equation}
Then the derived functor $(-)\stackrel{\bL}{\otimes}_{\dB}\wt{\mT}: \D(\dB)\to\D(\dA\underset{\mS}{\with}\dB)$
is fully faithful and it induces a semi-orthogonal decomposition $\D(\dA\underset{\mS}{\with}\dB)=\langle \D(\dE),\; \D(\dB)\rangle$
for some small DG category $\dE.$
\end{theorem}
\begin{proof}
The morphism $\phi$ induces a morphism $\wt{\phi}: \wt{\mS}\to \wt{\mT},$ the cone of which is quasi-isomorphic to
$\dB\hy\dC$\!--bimodule
$\underline{\mR}.$ The functor $(-)\stackrel{\bL}{\otimes}_{\dB}\underline{\mR}$ is isomorphic to the composition of $a^*$ and
$(-)\stackrel{\bL}{\otimes}_{\dB}\mR,$ while the functor $(-)\stackrel{\bL}{\otimes}_{\dB}\mT$ is isomorphic to $a_* (- \stackrel{\bL}{\otimes}_{\dB}\wt{\mT}).$
Thus, we obtain the following isomorphisms
\[
\begin{array}{lllll}
\Hom_{\D(\dC)}(\mM\stackrel{\bL}{\otimes}_{\dB}\underline{\mR},\; \mN\stackrel{\bL}{\otimes}_{\dB}\wt{\mT})
&
\cong
&
\Hom_{\D(\dA)}(\mM\stackrel{\bL}{\otimes}_{\dB}\mR,\; \mN\stackrel{\bL}{\otimes}_{\dB}\mT)
&
=
&0,
\\
\Hom_{\D(\dC)}(\mM\stackrel{\bL}{\otimes}_{\dB}\wt{\mS},\; \mN\stackrel{\bL}{\otimes}_{\dB}\underline{\mR})
&
\cong
&
\Hom_{\D(\dC)}(b^*(\mM),\; a^*(\mN\stackrel{\bL}{\otimes}_{\dB}\mR))
&
=
&0.
\end{array}
\]
Therefore for any DG  $\dB$\!--modules $\mM$ and $\mN$ we have isomorphisms
\[
\Hom_{\D(\dC)}(\mM\stackrel{\bL}{\otimes}_{\dB}\wt{\mT},\; \mN\stackrel{\bL}{\otimes}_{\dB}\wt{\mT})
\cong
\Hom_{\D(\dC)}(\mM\stackrel{\bL}{\otimes}_{\dB}\wt{\mS},\; \mN\stackrel{\bL}{\otimes}_{\dB}\wt{\mT})
\cong
\Hom_{\D(\dC)}(\mM\stackrel{\bL}{\otimes}_{\dB}\wt{\mS},\; \mN\stackrel{\bL}{\otimes}_{\dB}\wt{\mS}).
\]
The functor $(-)\stackrel{\bL}{\otimes}_{\dB}\wt{\mS}$ is fully faithful. Hence the functor $(-)\stackrel{\bL}{\otimes}_{\dB}\wt{\mT}$
is fully faithful too.
\end{proof}

\begin{theorem}\label{main2}
Let DG categories $\dA, \dB,$ DG $\dB\hy\dA$\!--bimodules $\mS, \mT, \mR,$ and a morphism $\phi$ be as above.
Assume that condition (\ref{cond}) holds. If the functors $(-)\stackrel{\bL}{\otimes}_{\dB}\wt{\mT}$ and
$\bR\Hom_{\dC}(\wt{\mT}, -)$ sends perfect modules to perfect ones, then there is
a semi-orthogonal decomposition of the form  $\prf(\dA\underset{\mS}{\with}\dB)=\langle \prf\dE,\; \prf\dB\rangle$
for a small DG category $\dE.$
\end{theorem}

\begin{definition} The category $\prfdg\dE$ (resp. $\prf\dE$) from Theorem \ref{main2} will called
Krull-Schmidt partner (KS partner) for $\prfdg\dA$ (resp. $\prf\dA$).
\end{definition}
Since the composition of $(-)\stackrel{\bL}{\otimes}_{\dB}\wt{\mT}$ and
$\bR\Hom_{\dC}(\wt{\mS}, -)$ is isomorphic to the identity,  K-theories of $\prfdg\dE$ and $\prfdg\dA$ are equal. Moreover, their K-motives  are isomorphic.


Let $X$ be a smooth projective scheme and let $\mP_s\in \prfdg X,\; s=1,2$ be two perfect complexes
such that their supports $\supp\mP_1$ and $\supp\mP_2$ do not meet.
Consider the gluing $\dD=\prfdg X\underset{\mS}{\oright}\prfdg\kk$ with $\mS=\mP_1\oplus\mP_2,$ and take $\mT=\mP_2.$
By Theorem \ref{main2} the functor $(-)\stackrel{\bL}{\otimes}_{\kk}\wt{\mT}$ from
$\prf\kk$ to $H^0(\dD)$ is fully faithful, and we obtain a semi-orthogonal decomposition
$H^0(\dD)=\langle \prf\dE, \prf\kk\rangle.$ Thus we get a KS partner $\prfdg\dE$ for $\prfdg X.$

For example, let $X=C$ be a smooth projective curve of genus $g=g(C),$ and let $\mP_1, \mP_2$ be torsion coherent sheaves of lengths
$l_s=\length P_s,\; s=1,2$  such that
$\supp\mP_1\cap\supp\mP_2=\emptyset.$
It can be easily checked that the KS partner $\prf\dE$ is not equivalent to $\prf C.$
Indeed, the integral bilinear form $\chi(E, F)=\sum_m (-1)^m \dim\Hom(E, F[m])$ on $K_0(C)$ goes through $\ZZ^2=H^{ev}(C, \ZZ).$
In this case the forms $\chi$ on $K_0(C)$ and $K_0(\dE)$ are equal to
\begin{equation}
\chi_C=\begin{pmatrix}
1-g & 1\\
-1 & 0
\end{pmatrix},\qquad{and}\qquad
\chi_{\dE}=\chi_t:=\begin{pmatrix}
t & 1\\
-1 & 0
\end{pmatrix}, \qquad\text{where}\quad t=1-g -l_1 l_2.
\end{equation}
The integral bilinear forms $\chi_t$ are not equivalent for different $t.$ Hence the categories $\prf C$ and $\prf\dE$ that are KS partners of $\prf C$ are not equivalent for different
$t=1-g-l_1l_2.$ The case of $\PP^1$ and two points $\mP_s=p_s,\; s=1,2$ is discussed in \cite[3.1]{O_q}.


\begin{thebibliography}{AKO}

\bibitem{Ke}B.~Keller, {\em Deriving DG categories,}
Ann. Sci. \'Ecole Norm. Sup. (4), {\bf 27} (1994), 63--102.

\bibitem{LO} V.~Lunts, D.~Orlov,
{\em Uniqueness of enhancement for triangulated categories}, J.
Amer. Math. Soc., {\bf 23} (2010), 3, 853--908.


\bibitem{O_glue} D.~Orlov, {\em Smooth and proper noncommutative schemes and gluing of DG
categories}, accepted to Adv. Math., http://arxiv.org/abs/1402.7364

\bibitem{O_q} D.~Orlov, {\em Geometric realizations of quiver algebras},
Tr. Mat. Inst. Steklova, {\bf 290} (2015), 80--94; transl. in
Proc. Steklov Inst. Math., {\bf 290} (2015), 70--83.

\bibitem{Ta} G.~Tabuada, {\em Th\'eorie homotopique des DG-cat\'egories,} These de l'Univ. Paris 7, (2007).


\end{thebibliography}
\end{document}